\documentclass{amsart}
\usepackage{amsmath,amsthm,amssymb}
\usepackage{tikz-cd}
\usepackage{thmtools}
\usepackage{enumitem}
\usepackage{quiver}
\usepackage[hidelinks]{hyperref}
\usepackage{accents}
\usepackage{enumitem}
\setlist{listparindent = \parindent, parsep=0pt,}
\setenumerate[1]{label = (\roman*), ref = (\roman*)}
\setenumerate[2]{label = (\alph*), ref = (\alph*)}
\usepackage[noabbrev,capitalise]{cleveref}
\crefformat{equation}{(#2#1#3)}
\crefformat{enumi}{#2#1#3}
\crefformat{enumii}{#2#1#3}

\title{Representation of abstract pseudogroups}
\author{J. L. Wrigley}
\thanks{
\paragraph{Acknowledgements}
I thank Mark Lawson for the warm encouragement to write this paper.  I acknowledge the support of Marie Sk{\l}odowska-Curie Grant No.\ 101273434.
}
\address[Joshua L. Wrigley]{Department of Mathematics and Statistics, Faculty of Sciences, Masaryk University, Kotlářská 2, 611 37 Brno, Czech Republic.}
\email{wrigley@math.muni.cz}
\urladdr{https://jlwrigley.github.io/}
\address{\includegraphics[height=1.0cm]{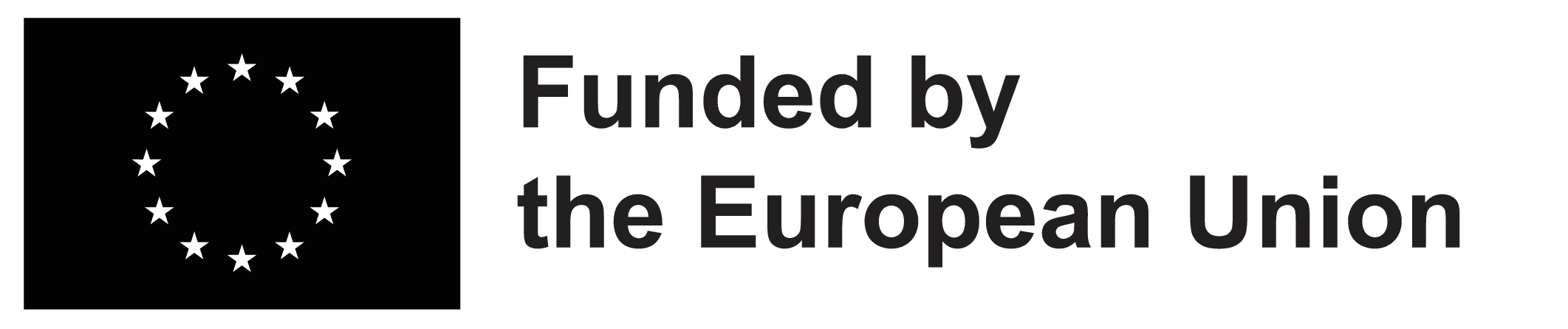}}
\theoremstyle{plain}

\newtheorem{thm}{Theorem}[section]
\newtheorem{lem}[thm]{Lemma}
\newtheorem{coro}[thm]{Corollary}
\newtheorem{prop}[thm]{Proposition}

\theoremstyle{definition}

\newtheorem{df}[thm]{Definition}

\newtheorem{rem}[thm]{Remark}

\newtheorem{ex}[thm]{Example}

\usepackage[
backend=biber,
sorting=nyt,
style=ieee
]{biblatex}
\renewcommand{\phi}{\varphi}

\newcommand{\id}{\text{id}}

\newcommand{\lrset}[2]{\left\{{#1}\,\middle\vert\,{#2}\right\}}
\newcommand{\sett}[2]{\{{#1}\,\vert\,{#2}\}}
\newcommand{\lrangle}[1]{\left\langle{#1}\right\rangle}

\newcommand{\Homeo}{\mathcal{H}}

\newcommand{\Idl}{\mathrm{Idl}}
\newcommand{\I}{\mathcal{I}}
\renewcommand{\H}{\mathcal{H}}
\renewcommand{\P}{\mathcal{P}}
\newcommand{\opens}{\mathcal{O}}
\newcommand{\src}{\mathsf{src}}
\newcommand{\tgt}{\mathsf{tgt}}
\newcommand{\nbhd}[1]{{{#1}}^\circ}
\newcommand{\eqr}[1]{\overline{{#1}}}
\newcommand{\2}{\mathbf{2}}

\usepackage[a4paper,
left=1.5in,
right=1.5in,
]{geometry}
\begin{document}
	
	\begin{abstract}
		The Wagner-Preston theorem states that every inverse semigroup can be embedded into an inverse semigroup of partial bijections on some set.  However, this embedding does not respect lattice operations in the natural ordering.  In this paper, we show that an abstract pseudogroup can be faithfully represented by partial bijections on a set if and only if its frame of idempotents is spatial.  More generally, we show that any abstract pseudogroup can be faithfully represented by partial homeomorphisms on a locale.
		
		\smallskip
		\noindent \textsc{\keywordsname.}{
			Pseudogroup, pseudogroup of transformations, inverse semigroup, Wagner-Preston theorem.
		}
		
		\smallskip
		\noindent \textsc{\subjclassname.}{
			20M18, 20M20, 20M30, 06D22, 18F70.
		}
	\end{abstract}
	\maketitle
	\section{Introduction}
	If group theory is the study of total symmetries, then inverse semigroup theory is the study of partial symmetries.  This slogan for inverse semigroup theory is justified by the Wagner-Preston theorem, proved independently on either side of the Iron Curtain in \cite{wagner} and \cite{preston}, which asserts that every inverse semigroup can be faithfully represented by partial bijections.  In construction and consequence, the Wagner-Preston theorem resembles Cayley's theorem for groups.
	
	However, the Wagner-Preston representation does not preserve many of the derived operations an inverse semigroup might possess.  These additional operations are natural ones to include when reasoning about partial symmetries, and formalise the following ways to combine partial symmetries:
	\begin{enumerate}
		\item The identity transformation is a partial symmetry, as is the partial symmetry whose domain is empty;
		\item Given a pair pair of partial symmetries $\alpha$ and $\beta$ with domains $U$ and $V$, then by restricting to the subset of $U \cap V$ on which $\alpha$ and $\beta$ agree, we obtain a new partial symmetry $\alpha \cap \beta$ that is both a restriction of $\alpha$ and $\beta$;
		\item Given a set of partial symmetries $\sett{\alpha_i}{i \in I}$, where each $\alpha_i$ has domain $U_i$, if each pair $\alpha_i, \alpha_j$ agree on the intersection $U_i \cap U_j$, then by `gluing' the partial symmetries together, we obtain a new partial symmetry $\alpha$ with domain $\bigcup_i U_i$.
	\end{enumerate}
	These additional operations are captured by the notion of an \emph{abstract pseudogroup} (as found in, e.g., \cite[\S 2.6]{lawson_txtbk}, and recalled in \cref{df:psdgrp}), which expresses these operations through a combination of algebraic and order-theoretic properties.
	
	The question we are concerned with in this paper is: when does an abstract pseudogroup admit a faithful representation by partial bijections?  We provide the following complete characterisation, thus giving a parallel of the Wagner-Preston theorem for abstract pseudogroups.
	\begin{thm}\label{mainthm:enough_pts}
		An abstract pseudogroup $A$ admits a faithful representation by partial bijections on a set if and only if the frame of idempotents has enough points.
	\end{thm}
	Recall that a frame `has enough points' if it is isomorphic to the lattice of open subsets of a space.  By combining the above representation result with the \emph{ideal completion} from \cite[\S 3.1]{lawson_lenz}, we will deduce that to check if an equation, involving only the finitary operations of a pseudogroup, holds in an arbitrary pseudogroup, it suffices to consider whether the equation holds for partial bijections (\cref{coro:part_bij_are_fin_complete}), formalising the practical intuition of pseudogroup theorists.  %For the interested reader, we note as well that partial bijections can be faithfully represented by partial isometries on a Hilbert space \cite[\S 4]{heunen}, and so combined with the above representation we deduce that every pseudogroup whose frame of idempotents has enough points admits a faithful representation by partial isometries.
	
	Although we cannot faithfully represent a pseudogroup by partial bijections if the frame of idempotents lacks enough points, we can nonetheless obtain a representation result by replacing any mention of sets by the next best thing: \emph{locales}, in the sense of `point-free' topology.
	\begin{thm}\label{mainthm:no_pts}
		Every pseudogroup $A$ admits a faithful representation by partial homeomorphisms on a locale.
	\end{thm}
	Our result completes some terminological housekeeping: initially, `pseudogroups' were defined `extensionally', as collections of homeomorphisms between open subsets of a topological space that are closed under certain operations (see \cref{ex:trans_psdgrp}); see \cite[\S I.1]{reinhart} and \cite[\S 2.8]{lawson_txtbk} for further discussion on the history of pseudogroups.  We follow the terminology of \cite{lawson_txtbk} and call these \emph{pseudogroups of transformations}.  In contrast to the extensional definition, abstract pseudogroups are defined `intensionally' as inverse semigroups with additional (partial) operations.  Our abstract pseudogroups are also called `abstract complete pseudogroups', for instance, in \cite{resende_adv}.  Our characterisation theorem therefore relates these two notions of `pseudogroup', showing that they coincide precisely when the frame of idempotents has enough points.
	
	We remark, for interest, that at least as far as the history of the ideas are concerned, the appearance of frame-theoretic concepts in \cref{mainthm:enough_pts} and \cref{mainthm:no_pts} represents a return to conceptual origin.  As explained in \cite{johnstone_pointless} and \cite[\S 2.8]{lawson_txtbk}, it was Ehresmann's work motivated by pseudogroups of transformations that, through his student Benabou, eventually led to the development of the frame-theoretic approach to point-free topology (see \cite{ehresmann_ger,benabou} for examples of this early development).  It is therefore perhaps unsurprising to see notions fundamental to point-free topology, such as `having enough points', reappearing in the study of abstract pseudogroups.
	
	Although the explicit statements of \cref{mainthm:enough_pts} and \cref{mainthm:no_pts} are seemingly missing in the literature, there are nonetheless enough clues in extant work to suggest that such representation results might be expected.  In particular, the particular representations we will construct in \cref{subsec:subquotient} and \cref{subsec:locale_of_subqts} bear a resemblance to the \'etale topological and localic groupoids associated with a pseudogroup, as studied in \cite{lawson_lenz,resende_notes,resende_adv}, a kinship that is deserving of further study.
	\subsection*{Overview}
	We proceeds as follows.
	First, we recall the definition of an abstract pseudogroup, and enough preliminary material regarding pseudogroups and frames to understand \cref{sec:spatial}, in which we prove our main result, \cref{mainthm:enough_pts}.
	
	We open \cref{sec:spatial} by observing that the condition identified in \cref{mainthm:enough_pts}, that the frame of idempotents has enough points, is a necessary condition (\cref{prop:psdgrp_emb_implies_enough_pts}).  Next, we construct two semigroup representations of a pseudogroup $A$ by partial bijections, with the first aiding in the construction of the latter.  The first semigroup representation is neither a pseudogroup representation nor faithful, but the second both preserves the pseudogroup structure and is faithful if the frame of idempotents of $A$ has enough points, completing the proof of \cref{mainthm:enough_pts}.  In \cref{subsec:dist}, as an application, we show that every \emph{distributive inverse semigroup} is faithfully represented by partial bijections, and thus deduce that to check if a (finitary) equation holds in an arbitrary pseudogroup, it suffices to look at partial bijections.
	
	Finally, in \cref{sec:localic_rep}, we describe what can be salvaged from \cref{mainthm:enough_pts} in the case where the frame of idempotents lacks enough points: we prove \cref{mainthm:no_pts}, that a faithful representation result is still possible using partial homeomorphisms of locales.  We recall, at the beginning of \cref{sec:localic_rep}, the additional locale-theoretic material we will need in the proof of \cref{mainthm:no_pts}.
	%%%%%%%%%%%%%%%%%%%%%%%%%%%%%%%%%%%%%%%%%%%%%%%%%%
	\section{Preliminaries}\label{sec:prelims}
	We first recall inverse semigroups and pseudogroups, before reviewing the aspects of frame theory we will use in \cref{sec:spatial}.  The reader is directed to \cite{lawson_txtbk} for more reading on the former, and \cite{picado_pultr,johnstone_stone} for more concerning frames.
	\subsection{Pseudogroups}
	An \emph{inverse semigroup} is a set $A$ with an associative binary operation ${\,\cdot\,} \colon A^2 \to A$ (we write $ab$ for $a \cdot b$) such that, for every $a \in A$, there is a unique element $a^{-1} \in A$ satisfying the equations $aa^{-1}a = a$ and $a^{-1}aa^{-1} = a^{-1}$.  We write $E(A) \subseteq A$ for the subset of \emph{idempotent} elements.  In an inverse semigroup, all idempotents commute (\cite[Theorem 1.2.8]{lawson_txtbk}).  For an element $a \in A$, we write $\src(a)$ for the idempotent $a^{-1}a$ and call this element the \emph{source} of $a$; similarly, we write $\tgt(a)$ for $aa^{-1}$ and call this the \emph{target}.  The reason for this terminology will become apparent in \cref{ex:part_bijs}.
	
	The \emph{natural order} is a partial order on any inverse semigroup, defined by $a \leqslant b$ if and only if $a = be$ for some idempotent $e \in E(A)$, equivalently $a = fb$ for a (different) idempotent $f \in E(A)$ (see \cite[Lemma 2.1.1]{lawson_txtbk} for further equivalent conditions).  As is standard, we will write $a < b$ if $a \leqslant b$ and $a \neq b$, and we will refer to this relation as the \emph{strict} natural order.  We are concerned with inverse semigroups where the natural order is particularly rich.
	\begin{df}\label{df:psdgrp}
		Let $A$ be an inverse semigroup with natural ordering $\leqslant$.
		\begin{enumerate}
			\item A subset $S \subseteq A$ is said to be \emph{jointly compatible} if, for any pair $a,b \in S$, both $ab^{-1}$ and $a^{-1}b$ are idempotent elements.
			\item We say that $A$ is an \emph{abstract pseudogroup} if every jointly compatible subset $S \subseteq A$ has a join $\bigvee S$ in the natural ordering $(A,\leqslant)$ and multiplication distributes over joins on both the left and the right, that is:
			\[\textstyle a\left(\bigvee S\right) = \bigvee aS \quad \text{and} \quad \left(\bigvee S\right)a = \bigvee Sa,\]
			where $aS = \sett{as}{s \in S}$ and $Sa = \sett{sa}{s \in S}$.  Since we will be primarily concerned with abstract pseudogroups, we will often drop the prefix `abstract'.
		\end{enumerate}
	\end{df}
	\begin{rem}
		In a pseudogroup, we ask for joins of jointly compatible subsets.  In fact, these are the only possible joins: if $A$ is an inverse semigroup with a subset $S \subseteq A$ for which the join $\bigvee S $ exists, then $S$ is jointly compatible (see \cite[Lemma 2.2.1]{lawson_txtbk}).
	\end{rem}
	We record some consequences of the definition of a pseudogroup:
	\begin{prop}\label{prop:consequences}
		Let $A$ be a pseudogroup.  
		\begin{enumerate}
			\item\label{enum:zero_exists} Then $A$ has a \emph{zero element} $0$, i.e.\ an element such that $0a = a0 = 0$ for all $a \in A$, and a \emph{unit} $1$, i.e.\ $A$ is a monoid.
			\item\label{enum:binary_meets_in_psdgrp} Additionally, $A$ has all binary meets with respect to the natural ordering and moreover meets distribute over joins, i.e.\ $a \land \bigvee_I b_i = \bigvee_I a \land b_i$.
		\end{enumerate}
	\end{prop}
	\begin{proof}
		For the zero element, take $0$ to be the join of the emptyset $\bigvee \emptyset$; then $a0 = a \left(\bigvee \emptyset\right) = \bigvee \emptyset = 0$ (and similarly $0a=0$).  For the unit, take $1$ to be the join $\bigvee E(A)$ over all the idempotents of $A$, which is clearly a jointly compatible subset.  We need to show that $a 1 = a\left(\bigvee E(A)\right) = \bigvee aE(a)$ is equal to $a$.  Note that each element $ae \in aE(a)$ is less than or equal to $a$, and so $\bigvee aE(A) \leqslant a$; conversely, $a = aa^{-1}a \in aE(a)$, and so $a \leqslant \bigvee aE(A)$.  Thus, $a1 = a$ (and similarly $1a=a$).  For \cref{enum:binary_meets_in_psdgrp}, see \cite{resende_inf_dist} or \cite[Lemmas 2.5.2 \& 2.5.4]{lawson_txtbk}.
	\end{proof}
	It is easily observed that for a pair of idempotents $e, f \in E(A)$, their meet $e \land f$ is precisely their product $ef$.
	\begin{ex}[Pseudogroup of partial bijections]\label{ex:part_bijs}
		Let $X$ be a set.  The partial bijections $\I(X)$ of $X$ is the archetypal example of a pseudogroup.  The pseudogroup structure on $\I(X)$ is most apparent when a partial bijection on $X$ is viewed in terms of its graph as a subset of $X^2$.
		
		Multiplication is given by relational composition.  The natural ordering is set-theoretic inclusion, or equivalently $f \leqslant g$ if $f$ is a restriction of $g$.  The unit is the identity function on $X$ while the zero element is the empty-set.  The meet of two partial bijections $f$ and $g$ is given by the intersection of their graphs $f \cap g$.  A subset of partial bijections $S \subseteq \I(X)$ is compatible if, given $f, g \in S$, both $f$ and $g$ agree on the intersection of their domain.  In this case, the join of $S$ is given by the set-theoretic union $\bigcup S$.
		
		For this reason, when discussing the pseudogroup $\I(X)$, we will denote a meet by intersection `$\cap$', and a join by union `$\bigcup$'.  Note also that a partial bijection is idempotent in $\I(X)$ if and only if it is a restriction of the identity, and so the subset of idempotents $E(\I(X))$ is order isomorphic to the powerset lattice $\P(X)$.  Under this isomorphism, $\src(f)$ is precisely the source, or domain, of the partial bijection $f$, and similarly for $\tgt(f)$.
	\end{ex}
	\begin{ex}[Pseudogroup of transformations]\label{ex:trans_psdgrp}
		We now give the classical definition of a pseudogroup (of transformations), as found, for instance, in \cite[Definition 1.1]{reinhart}.  Let $X$ be a topological space.  A \emph{pseudogroup of transformations over $X$} is a collection $A$ of homeomorphisms $a \colon U \xrightarrow{\sim} V$ between open subsets of $X$ such that:
		\begin{enumerate}
			\item $A$ is closed under composition and inverses of partial homeomorphisms;
			\item the identity on $X$ belongs to $A$;
			\item if $a \in A$, then the restriction of $a$ to any open subset of its domain is also contained in $A$;
			\item if $\sett{a_i}{i \in I} \subseteq A$ is a subset of partial homeomorphisms $a_i \colon U_i \xrightarrow{\sim} V_i$, where each $a_i$ and $a_j$ agree on the intersection of their domains, i.e.\ $S$ is jointly compatible, then $A$ contains a homeomorphism $a \in A$ whose domain is $\bigcup_I U_i$ and the restriction of $a$ to $U_i$ is precisely $a_i$.
		\end{enumerate}
		It is easy to check that $A$ is an abstract pseudogroup in the sense of \cref{df:psdgrp}, and that $E(A)$ is order isomorphic to the lattice $\opens(X)$ of open subsets of $X$.
		
		In particular, the collection of all homeomorphisms between open subsets of $X$ is a pseudogroup of transformations, which we denote by $\Homeo(X)$.
	\end{ex}
	\begin{ex}\label{ex:G+0}
		A group is an inverse semigroup but not a pseudogroup.  Each group $G$ yields a pseudogroup by augmenting $G$ with a zero element $0$, i.e.\ by taking the disjoint union $G + \{0\}$ with the multiplication
		\[
		g \cdot h : = \begin{cases}
			0 & \text{if $g = 0$ or $h=0$,} \\
			gh & \text{if $g, h \in G$.}
		\end{cases}
		\]		
		With this multiplication, $G+\{0\}$ is a pseudogroup; the natural order is easy to describe: $0 \leqslant g$ for all $g \in G$, and all other pairs $g, h \in G$ are incomparable.
	\end{ex}
	Even though the definition of a pseudogroup only asked for the existence of certain joins, the homomorphisms we will consider are required to preserve the additional structure we can derive from the definition of a pseudogroup (\cref{prop:consequences}).
	\begin{df}
		Given pseudogroups $A$ and $B$, a \emph{pseudogroup homomorphism} $\theta \colon A \to B$ is a function that preserves multiplication, the unit, all joins and binary meets in the natural ordering.
	\end{df}
	By a \emph{representation of a pseudogroup $A$ by partial bijections} (as in \cref{mainthm:enough_pts}), we mean an injective pseudogroup homomorphism $A \hookrightarrow \I(X)$ for some set $X$.  A \emph{pseudogroup of transformations} can then be understood as a pseudogroup equipped with such a faithful representation -- indeed, given a faithful representation $\theta \colon A \hookrightarrow \I(X)$ for a set $X$, then $A$ is a pseudogroup of transformations over the space where $X$ has been endowed with the topology whose open subsets are those subsets in the image of the restriction $\theta|_{E(A)} \colon E(A) \to E(\I(X)) \cong \P(X)$.
	\begin{rem}
		Note that our pseudogroup homomorphisms preserve binary meets, whereas the morphisms considered in \cite[\S 2.5]{lawson_txtbk} preserve only joins.
	\end{rem}
	Before continuing, we list some basic properties of inverse semigroups that we will use in \cref{sec:spatial}.  The reader may wish to retrospectively review these results in light of \cref{coro:part_bij_are_fin_complete}.
	\begin{prop}\label{prop:basic_props}
		Let $A$ be an inverse semigroup.
		\begin{enumerate}
			\item\label{enum:inv_of_product} For all $a, b \in A$, $(ab)^{-1} = b^{-1}a^{-1}$.
			\item\label{enum:a_inv_a_and_b} If $a \leqslant b$, then $b^{-1}a = a^{-1}a$.
			\item\label{enum:mult_dist_over_meet} If the meet $a \land b$ exists, then $(a \land b)c = ac \land bc$ and $c(a \land b) = ca \land cb$ for all $c \in A$.
			\item\label{enum:src_of_join} If the join $\bigvee_I b_i$ exists, then $(\bigvee_I b_i)^{-1} = \bigvee_I b_i^{-1}$ and $\src\left(\bigvee_I b_i\right) = \bigvee_I \src(b_i)$.
		\end{enumerate}
	\end{prop}
	\begin{proof}
		See \cite[Lemma 1.2.11]{lawson_txtbk} for \cref{enum:inv_of_product}.  For \cref{enum:a_inv_a_and_b}, since $a  \leqslant b$, there exists some idempotent $e$ such that $a = eb$, and so 
		\[a^{-1}a = b^{-1}e^{-1}eb = b^{-1}e^2b = b^{-1}eb = b^{-1}a.\]
		For \cref{enum:mult_dist_over_meet} and \cref{enum:src_of_join}, see \cite[Lemmas 2.3.3, 2.4.3 \& 2.4.4]{lawson_txtbk}.
	\end{proof}
	We will also make use of the following lemma.  Recall from \cite[\S 2.3]{lawson_txtbk} that an inverse semigroup $A$ is called a \emph{meet semigroup} if $(A,\leqslant)$ has binary meets, and a semigroup homomorphism is a \emph{meet semigroup} homomorphism if it preserves binary meets in addition to multiplication.
	\begin{lem}\label{lem:strict_ineq_implies_inject}
		Let $A$ and $B$ be meet semigroups.  Suppose that $\theta \colon A \to B$ is a meet semigroup homomorphism that preserves the strict natural ordering, that is
		\[
		a < b \implies \theta(a) < \theta(b)
		\]
		for all $a, b \in A$.  Then $\theta$ is injective.
	\end{lem}
	\begin{proof}
		If $a \neq b$, then either $a \land b < a$ or $a \land b < b$.  Without loss of generality, assume that $a \land b < a$.  Then $\theta(a) \land \theta(b) = \theta(a \land b) < \theta(a)$, and so $\theta(a) \neq \theta(b)$.
	\end{proof}
	\subsection{Frames}
	A pseudogroup in which every element is idempotent is more commonly known as a \emph{frame}, that is, a complete lattice $(A,\leqslant)$ satisfying the infinite distributivity law $a \land \bigvee_I b_i = \bigvee_I a \land b_i$ for all $a \in A$ and $\sett{b_i}{ i \in I} \subseteq A$.  A \emph{frame homomorphism} $f \colon A \to B$ is a map that preserves all joins and all finite meets, i.e.\ a pseudogroup homomorpism between pseudogroups of exclusively idempotent elements.
	\begin{ex}
		For any pseudogroup $A$, the set of idempotents $E(A)$, endowed with the restriction of the natural order, describes a frame.  In fact, for any inverse semigroup $B$, the idempotents $E(B)$ form a frame if and only if $B$ is a pseudogroup (\cite[Lemma 2.5.1]{lawson_txtbk}).  For any pseudogroup homomorphism $\theta \colon A \to B$, the restriction to idempotents $\theta|_{E(A)} \colon E(A) \to E(B)$ is a frame homomorphism.
	\end{ex}
	\begin{ex}
	Frames typically appear in the `point-free' approach to topology, where open subsets are treated as primitive, and points are a derived notion (for more on this philosophy, see \cite{johnstone_pointless}).  For any topological space $X$, the lattice of open subsets $\opens(X)$ is a frame, and similarly every continuous map $f \colon X \to Y$ between topological spaces defines a frame homomorphism $f^{-1} \colon \opens(Y) \to \opens(X)$ that maps an open subset of $Y$ to its inverse image.  This connection to topology will be pursued with greater earnest in \cref{subsec:frame_pres}.
	\end{ex}
	We now describe how points are recovered from the notion of a frame.  The frame of opens of the one-point space $\{\ast\}$ consists of just the bottom and top elements, and so it is often denoted by $\2 = \{0,1\}$.  As the points of a topological space $X$ correspond to continuous maps $\{\ast\} \to X$, by analogy we define a \emph{point} of a frame $A$ to be a frame homomorphism $p \colon A \to \2$.  As long as the space $X$ satisfies reasonable separation conditions, e.g.\ if $X$ is Hausdorff, the points of the frame $\opens(X)$ correspond to the points of $X$.
	
	Points can alternatively be described as \emph{completely prime filters} of $A$, i.e.\ subsets $F \subseteq A$ satisfying the following conditions:
	\begin{enumerate}
		\item $F$ is non-empty and upwards closed, i.e.\ if $a \in F$ and $a \leqslant b$, then $b \in F$;
		\item if $a$ and $b$ are both in $F$, then $a \land b \in F$;
		\item if $\bigvee_I a_i$ is contained in $F$, then $a_{i_0} \in F$ for some $i_0 \in I$.
	\end{enumerate}
	This description of points generalises the \emph{neighbourhood filter} of a point of a topological space.  Points as frame homomorphisms $p \colon A \to \2$ bijectively correspond to points as completely prime filters, as witnessed by sending $p$ to $p^{-1}(1) \subseteq A$.
	
	Although frames behave in many ways like topological spaces, there are strictly more frames than there are frames of opens on a topological space.
	\begin{df}
		A frame $A$ is \emph{spatial}, or \emph{has enough points}, if $A \cong \opens(X)$ for some topological space.
	\end{df}
	A non-spatial frame is described in \cite[Example C1.2.8]{elephant}.  Since frames are also the Lindenbaum-Tarski algebras for theories of \emph{geometric propositional logic} (see \cite[Remark D1.4.14]{elephant}), the existence of non-spatial frames can be understood as the failure of the `completeness theorem' for this logic.  We list some equivalent criteria for spatiality that we will use interchangeably below.
	\begin{prop}[cf.\ \S II.5 \cite{picado_pultr}]\label{prop:equiv_spatiality}
		The following are equivalent for a frame $A$:
		\begin{enumerate}
			\item $A$ is spatial;
			\item for every pair $a, b \in A$ with $a \neq b$, there is a point $p \colon A \to \2$ such that $p(a) \neq p(b)$;
			\item there is an injective frame homomorphism $A \hookrightarrow \P(X)$ to the powerset frame on some set $X$.
		\end{enumerate}
	\end{prop}	
	Further aspects of frame theory will be introduced when needed in \cref{subsec:frame_pres} and \cref{subsec:locale_prelims}.
	%%%%%%%%%%%%%%%%%%%%%%%
	\section{Faithful representation of pseudogroups with enough points}\label{sec:spatial}
	In this section, we prove our main theorem, \cref{mainthm:enough_pts} -- a characterisation of those pseudogroups with a faithful representation by partial bijections via the following property.
	\begin{df}
		We say that a pseudogroup \emph{has enough points} if the frame of idempotents $E(A)$ has enough points.
	\end{df}
	We can immediately see that this is a necessary condition as follows.
	\begin{prop}\label{prop:psdgrp_emb_implies_enough_pts}
		If a pseudogroup $A$ admits a faithful representation by partial bijections, then the frame of idempotents $E(A)$ has enough points.
	\end{prop}
	\begin{proof}
		Let $\rho \colon A \hookrightarrow \I(X)$ be such an injective pseudogroup homomorphism.  Restricting to idempotents, we obtain an injective frame homomorphism $\rho|_{E(A)} \colon E(A) \hookrightarrow E(\I(X))$.  Recall that the frame of idempotents of $\I(X)$ is order isomorphic to the powerset frame $\P(X)$.  Thus, as $E(A)$ embeds into a powerset frame, it has enough points (\cref{prop:equiv_spatiality}).
	\end{proof}
	\begin{rem}\label{ex:non_spatial_frame}
		Thus, any non-spatial frame gives an example of an abstract pseudogroup that cannot be faithfully represented by partial bijections.
	\end{rem}
	The majority of the remainder of this section is devoted to proving that having enough points is a sufficient condition; namely, we construct a faithful representation by partial bijections.  We first begin by describing the canonical action of a pseudogroup on the points of its frame of idempotents (which does not describe a faithful representation).  We then show how to augment each point by a subquotient of the pseudogroup $A$ to obtain the desired faithful representation.  The assumption that $A$ has enough points is only used in the final steps of the proof, in \cref{prop:rho_preserves_strict_ineq} and \cref{prop:rho_is_injective}.
	
	Finally, in \cref{subsec:dist} we explore some consequences of \cref{mainthm:enough_pts}, including the aforementioned faithful representation by partial bijections of any distributive inverse semigroup.
	\subsection{The action on points of the idempotents}\label{sec:action_on_points}
	Let $A$ be a pseudogroup and let $X$ denote the set of points of its frame of idempotents $E(A)$.  Recall that each point $x \in X$ is a completely prime filter $x \subseteq E(A)$, i.e.\ $x$ is a subset of idempotents.  However, it is beneficial to think of the elements of $x$ as the open neighbourhoods of the point.  To avoid confusion, for an idempotent $e \in E(A)$ we will write $\nbhd{e}$ for the subset $\sett{x}{e \in x} \subseteq X$, i.e.\ $x \in \nbhd{e}$ if and only if $e \in x$.
	\begin{df}
		Let $a$ be an element of $A$, and let $x \in \nbhd{\src(a)}$ be a point of $X$.  We denote by $a(x)$ the subset
		\[
		a(x) = {\uparrow}\sett{aea^{-1}}{e \in x} \subseteq E(A),
		\]
		that is the set of all idempotents $f \in E(A)$ for which $f \geqslant aea^{-1}$ for some $e \in x$, and we call this the \emph{image} of $x$ under $a$.
	\end{df}
	\begin{lem}
		For each point $x \in \nbhd{\src(a)} \subseteq X$, the subset $a(x) \subseteq E(A)$ is a completely prime filter of $E(A)$ that contains $\tgt(a)$, i.e.\ $a(x) \in \nbhd{\tgt(a)} \subseteq X$.
	\end{lem}
	\begin{proof}
		First note that $\{aea^{-1} \mid e \in x\}$ is in fact a subset of idempotents, via the calculation:
		\begin{align*}
			aea^{-1}aea^{-1} & = aa^{-1}ae^2a^{-1} & \text{(since idempotents commute),} \\
			& = aea^{-1} & \text{(using that $aa^{-1}a = a$ and $e^2 = e$)}.
		\end{align*}
		Second, note that as $a^{-1}a \in x$, we have $\tgt(a) = aa^{-1} = aa^{-1}aa^{-1} \in a(x)$.
		
		Next, we show that $a(x)$ is closed under meets, which recall that for idempotents is the same thing as being closed under multiplication.  If $f \geqslant aea^{-1}$ and $f' \geqslant ae'a^{-1}$, where $e$ and $e'$ are both contained in $x$, then $ee' \in x$.  Hence, by a similar calculation as above, we have that $ff' \geqslant aea^{-1}ae'a^{-1} = a ee' a^{-1}$, and so $ff' \in a(x)$ too.
		
		Finally, we show that if $\bigvee_I f_i \in a(x)$, then $f_{i_0} \in x$ for some $i_0 \in I$.  Suppose that $aea^{-1} \leqslant \bigvee_I f_i $ for some $e \in x$.  Then we have $a^{-1}aea^{-1}a \leqslant a^{-1}\left(\bigvee_I f_i\right)a$.  Note that $a^{-1}aea^{-1}a = a^{-1}ae$ is contained in $x$, since $a^{-1}a \in x$, $e \in x$, and $x$ is closed under multiplication.  Hence, $ a^{-1}\left(\bigvee_I f_i\right)a$ is contained in $x$ as well.  By distributivity of multiplication with joins, we have
		\[
		\textstyle  a^{-1}\left(\bigvee_I f_i\right)a = \bigvee_I a^{-1}f_ia, 
		\]
		and so, as $x$ is completely prime, $a^{-1}f_{i_0}a$ is contained in $x$ for some $i_0 \in I$.  Thus, as $a a^{-1}f_{i_0}aa^{-1} \in a(x)$ and $a a^{-1}f_{i_0}aa^{-1}  \leqslant f_{i_0}$, we have that $f_{i_0} \in x$ as desired.
	\end{proof}
	Thus, each element $a$ of $A$ acts on the points of $E(A)$ by a partial map:
	\[
	\tau_a \colon \nbhd{\src(a)} \to \nbhd{\tgt(a)}.
	\]
	Note that, by taking the inverse, we also have a map $\tau_{a^{-1}} \colon \nbhd{\tgt(a)} \to \nbhd{\src(a)}$. 
	\begin{lem}
		The map $\tau_{a^{-1}}$ is inverse to $\tau_a \colon \nbhd{\src(a)} \to \nbhd{\tgt(a)}$; in particular, each $\tau_a$ is a partial bijection on $X$.
	\end{lem}
	\begin{proof}
		Let $x \in \nbhd{\src(a)}$.  Observe that 
		\begin{align*}a^{-1}a(x) & = {\uparrow}\sett{a^{-1}e'a}{e' \geqslant aea^{-1} \text{ where } e \in x}, \\
			& = {\uparrow}\sett{a^{-1}aea^{-1}a}{e \in x}.
		\end{align*}
		Since $x$ contains $a^{-1}a$ and is both closed under multiplication and upwards closed, we thus have that ${\uparrow}\sett{a^{-1}aea^{-1}a}{e \in x} = {\uparrow}\sett{e}{e \in x} = x$, i.e.\ $a^{-1}a(x) = x$ as desired.  By symmetry, $aa^{-1}(y) = y$ for all $y \in \nbhd{\tgt(a)}$, completing the proof.	
	\end{proof}
	\begin{prop}\label{prop:tau_is_homomorphism}
		Let $A$ be a pseudogroup and $X$ the set of points of the frame of idempotents $E(A)$.  The assignment $\tau \colon A \to \I(X)$ is a semigroup homomorphism.
	\end{prop}
	\begin{proof}
		It remains to show that $\tau_{ab}$, for $a, b \in A$, is the composite of the partial bijections $\tau_a$ and $\tau_b$.  We first show that the domains of $\tau_{ab}$ and $\tau_a \tau_b$ coincide; recall that the domain of the former is $\sett{x \in X}{x \in \src(ab)}$ while the domain of the latter is the set $\sett{x \in X}{a(x) \in \nbhd{\src(b)}}$.  Suppose that $x \in X$ is contained in $\nbhd{\src(ab)}$, i.e.\ $b^{-1}a^{-1}ab \in x$, then $bb^{-1}a^{-1}abb^{-1} \in b(x)$, and since $bb^{-1}a^{-1}abb^{-1} \leqslant a^{-1}a$, we have $b(x)\in \nbhd{\src(a)}$ as desired.  Conversely, suppose that $b(x) \in \nbhd{\src(a)}$, so that $a^{-1}a \in b(x)$, in addition to $bb^{-1} \in b(x)$.  Thus, $bb^{-1}a^{-1}a \in b(x)$ and so $ b^{-1}b(x) = x$ contains $b^{-1}bb^{-1}a^{-1}ab = b^{-1}a^{-1}ab$, i.e.\ $x \in \nbhd{\src(ab)}$ as desired.  Now that we have that the domains coincide, it is immediate to see that $\tau_a \tau_b$ and $\tau_{ab}$ act in the same way on points in the domain.
	\end{proof}
	Note that the semigroup homomorphism $\tau \colon A \to \I(X)$ does not need to be a pseudogroup homomorphism nor faithful, and so it cannot serve alone as a representation of the pseudogroup $A$ by partial bijections.  For example, consider the case where $A$ is taken to be the pseudogroup $G + \{0\}$ induced from a (non-trivial) group $G$ (as in \cref{ex:G+0}).  In this case, the idempotents form the two-element frame, which has a unique point, which we denote by $\ast$.  Every non-zero element of $G + \{0\}$ has source and target $\ast$, and so every non-zero element of $G + \{0\}$ is sent by $\tau$ to the same map (the identity on $\ast$).  This is clearly not faithful, and moreover it can be checked that $\tau$ does not preserve meets.
	%%%%%%%%%%%%%%%%%%%%%%%%%%%%%%%%%%%%%%%%%%%%%%%%%%%%%%%%%%%
	\subsection{Adding a subquotient}\label{subsec:subquotient}
	We now rectify these defects to obtain a genuine pseudogroup embedding $A \hookrightarrow \I(Y)$ for some larger set $Y$.  Our strategy is to `duplicate' the points of $X$ so that we are able to distinguish between elements of $A$ who have the same source and target, thus avoiding the problem encountered in \cref{sec:action_on_points}.  As a result, our space $Y$ will automatically come equipped with a projection $q \colon Y \to X$; indeed, it is easier to intuit our construction of $Y$ in terms of the fibres of this map.  Informally, for each point $x \in X$, we wish the fibre $q^{-1}(x) \subseteq Y$ to represent all the other points of $X$ that map to $x$ using one of the partial bijections $\tau_a$, where $a \in A$.  This is captured by the following construction.
	\begin{df}\label{df:underlying_set}
		Let $A$ be a pseudogroup.
		\begin{enumerate}
			\item For each point $x$ of its frame of idempotents, we define $\sim_x$ to be the partial equivalence relation on $A$ given by $a \sim_x a'$ if and only if $x \in \nbhd{\tgt(a \land a')}$.
			\item Let $Y$ be the set consisting of pairs $(x,\eqr{a})$, where $\eqr{a}$ is an equivalence class for the partial equivalence relation $\sim_x$.
		\end{enumerate}
	\end{df}
	By a partial equivalence relation, we mean a relation on $A$ that is symmetric and transitive.  Thus, each partial equivalence relation $\sim_x$ defines an equivalence relation on the subset $\sett{a}{a \sim_x a} \subseteq A$, which is precisely the subset $\sett{a}{x \in \nbhd{\tgt(a)}} \subseteq A$.
%	\begin{lem}
%		The relation $\sim_x$ is, in fact, a partial equivalence relation.
%	\end{lem}
%	\begin{proof}
%		Symmetry is immediate, while transitivity follows from the inequality 
%		\begin{equation}\label{eq:trans}
%			\tgt(a \land a')\tgt(a' \land a'') \leqslant \tgt(a \land a' \land a'')
%		\end{equation}
%		(which is in fact an equality, as is easily shown), since $ \tgt(a \land a' \land a'') \leqslant \tgt(a \land a'')$ by \cite[Proposition 2.1.2]{lawson_txtbk}.  To show \cref{eq:trans}, we have that
%		\begin{align*}
%			& \tgt(a \land a')\tgt(a' \land a'')(a \land a') \land \tgt(a \land a')\tgt(a' \land a'')(a' \land a'')  \\
%			= \ & \tgt(a \land a')\tgt(a' \land a'')(a \land a' \land a'') \\
%			\leqslant \ & (a \land a' \land a''),
%		\end{align*}
%		from which, using \cite[Lemma 1.2.13 \& Proposition 2.1.2]{lawson_txtbk}, we deduce that
%		\begin{align*}
%		& \tgt(\tgt(a \land a')\tgt(a' \land a'')(a \land a') \land \tgt(a \land a')\tgt(a' \land a'')(a' \land a''))  \\
%		\leqslant \ & \tgt(\tgt(a \land a')\tgt(a' \land a'')\tgt(a \land a'))\tgt(\tgt(a \land a')\tgt(a' \land a'')\tgt(a' \land a'')) \\
%		= \ & 
%		\end{align*}
%	\end{proof}
	\begin{rem}
		The above partial equivalence relation should be familiar: it is modelled on the germ construction found in \cite[\S I]{resende_notes}.
	\end{rem}
	\begin{df}
		For each element $b \in A$, we define $\rho_b \colon Y \rightharpoondown Y$ as the partial function whose domain is the subset $\sett{(x,\eqr{a})}{x \in\nbhd{ \src(b)}} \subseteq Y$, and which acts by
		\[(x,\eqr{a}) \mapsto (b(x),\eqr{ba}).\]
	\end{df}
	\begin{lem}
		The partial function $\rho_b \colon Y \rightharpoondown Y$ is well-defined, i.e.\ if $a \sim_x a'$ then $ba \sim_{b(x)} ba'$, and moreover it is inverse to $\rho_{b^{-1}}$.
	\end{lem}
	\begin{proof}
		Suppose that $x \in \nbhd{\tgt(a \land a')}$, i.e.\ $(a \land a')(a \land a')^{-1} \in x$.  Via the calculation
		\begin{align*}
			(ba \land ba')(ba \land ba')^{-1} & = (b(a \land a'))(b(a \land a'))^{-1} & \text{(using \cref{prop:basic_props}\cref{enum:mult_dist_over_meet}),} \\
			& = b(a \land a')(a \land a')^{-1} b^{-1} & \text{(using \cref{prop:basic_props}\cref{enum:inv_of_product}),}
		\end{align*}
		we deduce that $b(x) \in \nbhd{\tgt(ba \land ba')}$.
		
		Next, we need to show that $\rho_{b^{-1}} \circ \rho_b$ acts as the identity on all $(x,\eqr{a})$ with $x \in \nbhd{\src(b)}$; that is, we wish to show that $(x,\eqr{a}) = (b^{-1}b(x),\eqr{b^{-1}ba})$.  Equality in the first component, $x = b^{-1}b(x)$, follows from \cref{prop:tau_is_homomorphism}, so it remains to show that $a \sim_x b^{-1}ba$, i.e.\ $x \in \nbhd{\tgt(b^{-1}ba \land a)}$.  The first simplification to be made is that it suffices to show $x \in \nbhd{\tgt(b^{-1}ba)}$, as clearly $b^{-1}ba \leqslant a$ and so $b^{-1}ba \land a = b^{-1}ba$.  Note that, since $x \in \nbhd{\tgt(a)}$ in addition to $x \in \nbhd{\src(b)}$, we have that $aa^{-1} \in x$ and $b^{-1}b \in x$, and so, as $x$ is closed under multiplication, $ b^{-1}baa^{-1}b^{-1}b \in x$, i.e.\ $x \in \nbhd{\tgt(b^{-1}ba)}$ as desired.  By swapping $b^{-1}$ and $b$ in the above argument, we obtain the reverse equality: $\rho_b \circ \rho_{b^{-1}} (x,\eqr{a}) = (x,\eqr{a})$ for all $x \in \nbhd{\src(b^{-1})}$.
	\end{proof}
	%The first of the two main results of this paper is that the above defined assignment is the representation we seek.
	\begin{thm}\label{thm:rho_is_psdgrp_embedding}
		If $A$ is a pseudogroup with enough points, the assignment $\rho \colon A \to \I(Y)$ is an injective pseudogroup homomorphism.
	\end{thm}
	We break down the proof of this theorem into the following smaller steps: we first show that $\rho$ preserves multiplication and the unit, then that $\rho$ preserves binary meets, then that joins are also preserved, and finally that $\rho$ is injective when $A$ has enough points.
	\begin{prop}\label{prop:rho_is_semigrp_homo}
		The assignment $\rho \colon A \to \I(Y)$ is a unit-preserving semigroup homomorphism.
	\end{prop}
	\begin{proof}
		 We need to show that $\rho_{bb'} = \rho_b \rho_{b'}$, where the latter denotes the composition as partial bijections.
		 %, i.e.\ the composite $\rho_b \circ \rho_{b'}|_{\rho_{{b'}}^{-1}(\src(\rho_b))}$ where we have restricted the domain of $\rho_{b'}$ to those pairs $(x,\eqr{a})$ whose image under $\rho_{b'}$ lands in the domain of $\rho_b$.  
		 It is clear that if $(x,\eqr{a})$ is in the domain of $\rho_{bb'}$ and $\rho_b \rho_{b'}$, then $\rho_{bb'}(x,\eqr{a}) = (bb'(x),\eqr{bb'a}) = \rho_b \rho_{b'}(x,\eqr{a})$, and so it suffices to check that the domains of $\rho_{bb'}$ and $\rho_b \rho_{{b'}}$ agree.  This follows from the proof given in \cref{prop:tau_is_homomorphism}.
		 
		 Finally, it is easy to check that $\rho_1(x,\eqr{a}) = (1(x),\eqr{1a}) = (x,\eqr{a})$ and $x \in \nbhd{\src(1)}$ for all $x \in X$, i.e.\ that the unit is preserved.
	\end{proof}
	\begin{coro}[Lemma 2.1.6 \cite{lawson_txtbk}]\label{coro:monotone}
		The function $\rho \colon A \to \I(Y)$ is monotone with respect to the partial orderings.
	\end{coro}
	\begin{lem}\label{lem:rho_preserves_meets}
		The semigroup homomorphism $\rho \colon A \to \I(Y)$ preserves binary meets.
	\end{lem}
	\begin{proof}
		We aim to show that $\rho_b \cap \rho_{b'} = \rho_{b \land b'}$ for all $b , b' \in A$.  Since $\rho$ is monotone, one inclusion $\rho_{b \land b'} \subseteq \rho_b \cap \rho_{b'}$ is immediate.  To prove the converse, it suffices to show that the domain of $\rho_{b \land b'}$ contains the domain of $\rho_b \cap \rho_{b'}$.  Recall that an element $(x,\eqr{a})$ is in the domain of $\rho_b \cap \rho_{b'}$ if it is in the domain of $\rho_b$ and $\rho_{b'}$ and, moreover, $\rho_b(x,\eqr{a}) = \rho_{b'}(x,\eqr{a})$, that is, $(b(x),\eqr{ba}) = (b'(x),\eqr{b'a})$. 
		
		Suppose we are given such a pair $(x,\eqr{a})$.  Using that $(x, \eqr{a}) = \rho_{b^{-1}}\rho_b(x,\eqr{a})$, we have that $(x,\eqr{b^{-1}ba}) = (x,\eqr{b^{-1}b'a})$, and so $x \in \nbhd{\tgt(b^{-1}ba \land b^{-1}b'a)}$, or equivalently, by \cref{prop:basic_props}\cref{enum:mult_dist_over_meet}, $x \in\nbhd{ \tgt(b^{-1}(b \land b')a)}$; that is, we have $b^{-1}(b \land b')aa^{-1}(b \land b')^{-1}b \in x$.  Applying the equality $b^{-1}(b \land b') = (b \land b')^{-1}(b \land b')$ from \cref{prop:basic_props}\cref{enum:a_inv_a_and_b}, we obtain 
		\begin{align*}
			b^{-1}(b \land b')aa^{-1}(b \land b')^{-1}b & = (b \land b')^{-1}(b \land b')aa^{-1}(b \land b')^{-1}(b \land b') \\
			& = aa^{-1}(b \land b')^{-1}(b \land b')
			\leqslant (b \land b')^{-1}(b \land b').
		\end{align*}
		Thus $x \in \nbhd{\src(b \land b')}$, i.e.\ $(x, \eqr{a}) $ is in the domain of $\rho_{b\land b'}$ as desired.
	\end{proof}
	\begin{lem}\label{lem:rho_preserves_joins}
		The semigroup homomorphism $\rho \colon A \to \I(Y)$ preserves those joins that exist in $A$.
	\end{lem}
	\begin{proof}
		Let $\{b_i \mid i \in I\} \subseteq A$ be a jointly compatible subset of $A$, and so the join $\bigvee_I b_i$ exists in $A$.  We aim to show that $\rho_{\bigvee_I b_i} = \bigcup_I \rho_{b_i}$.  As $\rho$ is monotone, one inclusion $\bigcup_I \rho_{b_i} \subseteq \rho_{\bigvee_I b_i}$ is immediate.  It remains to show that every element $(x,\eqr{a})$ in the domain of $\rho_{\bigvee_I b_i}$ is in the domain of $\rho_{b_{i_0}}$, for some $i_0 \in I$.
		%and moreover $\rho_{b_{i_0}}(x,\eqr{a}) = \rho_{\bigvee_I b_i}(x,\eqr{a})$.  
		
		Let $(x,\eqr{a})$ be in the domain of $\rho_{\bigvee_I b_i}$.  Recall from \cref{prop:basic_props}\cref{enum:src_of_join} that $\src(\bigvee_I b_i) = \bigvee_I b_i^{-1}b_i$, and so $\bigvee_I b_i^{-1}b_i \in x$.  As $x$ is completely prime, $b_{i_0}^{-1}b_{i_0} \in x$ for some $i_0 \in I$, and so $(x,\eqr{a})$ is in the domain of $\rho_{b_{i_0}}$.  
	\end{proof}
	The previous results, \cref{prop:rho_is_semigrp_homo}, \cref{lem:rho_preserves_meets} and \cref{lem:rho_preserves_joins}, did not require that $A$ has enough points; we only invoke this hypothesis now for the last step.
	\begin{lem}\label{prop:rho_preserves_strict_ineq}
		If $A$ has enough points, then $\rho \colon A \to \I(Y)$ preserves the strict natural order.
	\end{lem}
	\begin{proof}
		Suppose that $a < b$, so that $a^{-1}a < b^{-1}b$.  Then since the frame of idempotents $E(A)$ has enough points, there is some point $x$ of $E(A)$ that separates $a^{-1}a$ and $b^{-1}b$, i.e.\ $x \in b^{-1}b$ and $x \not \in a^{-1}a$.  Thus, $(x,\eqr{b^{-1}})$ is a pair contained in the domain of $\rho_b$, but not in the domain of $\rho_a$.  Hence, $\rho_a \subsetneq \rho_b$ as desired.
	\end{proof}
	\begin{coro}\label{prop:rho_is_injective}
		If $A$ has enough points, the pseudogroup homomorphism $\rho \colon A \to \I(Y)$ is injective.
	\end{coro}
	\begin{proof}
		By combining \cref{lem:rho_preserves_meets} and \cref{prop:rho_preserves_strict_ineq} with \cref{lem:strict_ineq_implies_inject}.
	\end{proof}
	Combined together, \cref{prop:rho_is_semigrp_homo}, \cref{lem:rho_preserves_meets}, \cref{lem:rho_preserves_joins} and \cref{prop:rho_is_injective} constitute a proof of \cref{thm:rho_is_psdgrp_embedding}.  Now the two statements of \cref{thm:rho_is_psdgrp_embedding} and \cref{prop:psdgrp_emb_implies_enough_pts} together entail \cref{mainthm:enough_pts}, our characterisation of pseudogroups of transformations.
	\begin{ex}\label{ex:rho_for_frame}
		If $A$ is a pseudogroup in which every element is idempotent, i.e.\ $A$ is a frame, then in fact the set $Y$ described in \cref{df:underlying_set} and the set $X$ of points of $E(A) = A$ are isomorphic.  It is easily checked that, for each $x \in X$ and $a \in A$, $a \sim_x a$ if and only if $a \in x$, and $a \sim_x a'$ for all $a, a' \in x$; that is to say, there is a unique equivalence class for the partial equivalence relation $\sim_x$, and so $X \cong Y$.
		
		Clearly, modulo this isomorphism, the representations $\tau \colon A \to \I(X)$ and $\rho \colon A \to \I(Y)$ from \cref{prop:tau_is_homomorphism} and \cref{thm:rho_is_psdgrp_embedding} are identical; the representation acts by sending an element $a \in A$ to the restriction of the identity on $X$ to the subset of those points $x \in \nbhd{a}$.  In other words, the representation is the composition of the canonical frame homomorphism $A \to \P(X)$ from a frame to the powerset frame on its points with the inclusion $\P(X) \subseteq \I(X)$ of the idempotents (\cite[\S II.4-5]{picado_pultr}).  Recall that the map $A \to \P(X)$ is injective if and only if $A$ is spatial.
	\end{ex}
	\begin{ex}\label{ex:rho_for_group}
		Consider now the pseudogroup $G + \{0\}$ derived from a group $G$, as described in \cref{ex:G+0}.  Recall that, in this setting, there is a unique point $\ast$ of the frame of idempotents.  The partial equivalence relation $\sim_\ast$ on $G + \{0\}$ is defined on the subset $G \subseteq G + \{0\}$, on which it is trivial, i.e.\ $g \sim_x h$ if and only if $g = h$.  Thus, the set $Y$ defined in \cref{df:underlying_set} is isomorphic to $G$.
		
		It follows that $\rho \colon G + \{0\} \to \I(G)$ sends a group element $g \in G$ to the total bijection $\rho_g \colon G \to G$ given by $\rho_g(h) = gh$; that is, the restriction of $\rho$ to $G \subseteq G + \{0\}$ is precisely the representation of $G$ by total bijections on a set found in Cayley's theorem.  The remaining element, $0 \in G + \{0\}$, is sent by $\rho$ to the empty partial bijection on $G$.
	\end{ex}
	\subsection{Faithful representation of distributive inverse semigroups}\label{subsec:dist}
	The derived operations on partial symmetries captured by the notion of pseudogroup are infinitary in nature, in that we can take joins of compatible subsets of any arity.  However, there is also a natural finitary variant, where we only allow for finitary joins.  This is captured by the notion of an \emph{distributive inverse semigroups}, which first appeared in \cite[\S 2]{kaarli_marki}.  We recall here a modified version of the definition, before applying \cref{mainthm:enough_pts} to show that every such distributive inverse semigroup admits a faithful representation by partial bijections.
	\begin{df}\label{df:dist_inv_smgrp}
		A \emph{distributive inverse semigroup} $A$ is an inverse semigroup satisfying the following:
		\begin{enumerate}
			\item $A$ has both a zero element and a unit,
			\item $A$ has all binary meets with respect to the natural order,
			\item and each \emph{compatible pair}, i.e.\ elements $a, b \in A$ where $ab^{-1}$ and $a^{-1}b$ are idempotent, has a join in the natural order, and multiplication distributes over these joins on the left and right.
		\end{enumerate}
		A homomorphism of distributive inverse semigroups $A \to B$ is a semigroup homomorphism that preserves the above structure.
	\end{df}
	\begin{rem}
		Note that \cite[\S 2]{kaarli_marki} does not ask for binary meets in their definition of a distributive inverse semigroup, while the definition found in \cite[\S 2.6]{lawson_txtbk} asks only for a zero element and binary joins of compatible pairs.  The theory we develop can be made to work in this more general setting, but for convenience we assume existence of binary meets, etc.
	\end{rem} 
	In particular, every pseudogroup is a distributive inverse semigroup, and pseudogroup homomorphisms are distributive inverse semigroup homomorphisms, i.e.\ there is a forgetful functor from pseudogroups to distributive inverse semigroups.  As shown in \cite{lawson_lenz}, this forgetful functor has a left adjoint: there is a free completion to a pseudogroup for any distributive inverse semigroup.
	\begin{prop}[Proposition 3.1 \cite{lawson_lenz}]\label{prop:ideal_compl}
		For every distributive inverse semigroup $A$, there is an injective homomorphism of distributive inverse semigroups $\sigma \colon A \hookrightarrow \Idl(A)$.
	\end{prop}
	The elements of the pseudogroup $\Idl(A)$ consist of the downward closed, jointly compatible subsets $S \subseteq A$ which are $\lor$-\emph{closed}, in the sense that if $a, b \in S$ then $a \lor b \in S$ too. Such a subset $S \subseteq A$ is called an \emph{ideal}, in analogy with the ideals encountered in order theory (\cite[Definition 2.20]{davey_priestley}).  The embedding $\sigma$ sends an element $a \in A$ to the down segment ${\downarrow} a = \sett{b \in A}{b \leqslant a} \subseteq A$ of $a$ in the natural order.  Although the fact is not explicitly mentioned in \cite[\S 3.1]{lawson_lenz} or \cite[\S 2.7]{lawson_txtbk}, in which the ideal construction is discussed, it follows by direct inspection that the unit $1 \in A$ is preserved, i.e.\ ${\downarrow}1$ is a unit for $\Idl(A)$, and similarly $\sigma$ preserves binary meets, i.e.\ ${\downarrow} a \cap {\downarrow} b = {\downarrow}(a \land b)$. (Thus, even though our distributive inverse semigroups possess more structure than those considered in \cite{lawson_lenz} or \cite{lawson_txtbk}, the same construction still applies.)

	Observe that the frame of idempotents $E(\Idl(A))$ is also the \emph{ideal completion} of the distributive lattice $E(A)$, i.e.\ $E(\Idl(A)) = \Idl(E(A))$ (this follows from the identification of idempotents of $\Idl(A)$ found in the proof of \cite[Theorem 2.5.6]{lawson_txtbk}), and so $E(\Idl(A))$ is a \emph{coherent frame} in the terminology of \cite[\S II.3.2]{johnstone_stone}.  Assuming a weak choice principle, at least as strong as the ultrafilter lemma, the frame $\Idl(E(A))$ is spatial (\cite[Theorem II.3.4]{johnstone_stone}).  Thus, by \cref{mainthm:enough_pts}, $\Idl(A)$ admits a pseudogroup embedding $\Idl(A) \hookrightarrow \I(X)$ for some set $X$.  Composing this embedding $\Idl(A) \hookrightarrow \I(X)$ with the ideal completion $A \hookrightarrow \Idl(A)$ from \cref{prop:ideal_compl} yields the following.
	\begin{coro}\label{coro:repr_dist_invsgrp}
		Every distributive inverse semigroup $A$ admits a faithful representation by partial bijections $\rho \colon A \hookrightarrow \I(X)$ on some set $X$.
	\end{coro}
	An immediate consequence of \cref{coro:repr_dist_invsgrp} is the following, which formalises the common intuition of the practising semigroup theorist -- if an equation holds for all partial bijections, then it holds for all elements in a pseudogroup/distributive inverse semigroup.
	\begin{thm}\label{coro:part_bij_are_fin_complete}
		Let $\phi(z_1, \dots , z_n)$ and $\psi(z_1, \dots , z_n)$ be terms expressible in the language $\{\,\cdot\,, 0, 1, \land , \lor,{\, }^{-1}\}$, where $z_1, \dots , z_n$ are free variables.  Suppose that for any set $X$, the pseudogroup satisfies the sentence 
		\[\forall \, f_1 , \dots , f_n \in \I(X) . \, \phi(f_1, \dots , f_n) \leqslant \psi(f_1, \dots , f_n).\] 
		Then any distributive inverse semigroup $A$ also satisfies the sentence 
		\[\forall \, a_1, \dots , a_n \in A. \, \phi(a_1, \dots , a_n) \leqslant \psi(a_1, \dots , a_n).\]
	\end{thm}
	\begin{proof}
		By \cref{coro:repr_dist_invsgrp}, any distributive inverse semigroup $A$ admits faithful representation by partial bijections $\rho \colon A \hookrightarrow \I(X)$, i.e.\ $\rho$ is an injective map that preserves the operations $\{\,\cdot\,,0,1,\land,\lor,{\, }^{-1}\}$.  As a result, $\rho(\phi(a_1, \dots , a_n)) = \phi(\rho(a_1), \dots , \rho(a_n))$, for all $a_1, \dots , a_n \in A$, and similarly for $\psi$.  Therefore, by assumption, we have that
		\begin{align*}
			\rho(\phi(a_1, \dots , a_n)) & = \phi(\rho(a_1), \dots , \rho(a_n)) \\
			 & \leqslant  \psi(\rho(a_1), \dots , \rho(a_n)) = \rho(\psi(a_1, \dots , a_n)).
		\end{align*}
		Since $\rho $ is injective, and hence an order embedding, we deduce that $\phi(a_1, \dots , a_n) \leqslant \psi(a_1, \dots , a_n)$ as desired.
	\end{proof}
	\begin{rem}
		The distributive inverse semigroups described in \cref{df:dist_inv_smgrp} can be axiomatised by an \emph{essentially algebraic theory} (see \cite[\S 3.D]{adamek_rosicky}; essentially algebraic theories are equivalent to the \emph{`cartesian'} theories of \cite[Definition D1.3.4]{elephant}).  \cref{coro:part_bij_are_fin_complete} therefore expresses that the essentially algebraic theory of distributive inverse semigroups is determined by pseudogroups of the form $\I(X)$, just as the variety of distributive lattices is generated by powerset lattices (\cite[Theorem 10.21]{davey_priestley}).
	\end{rem}
	As an example application of \cref{coro:part_bij_are_fin_complete}, we deduce certain inequalities in a pseudogroup that we will use in the next section.  Although the inequalities below have (choice-free) algebraic proofs, it is nonetheless simpler, and more intuitive, to deduce their validity by looking at partial bijections on a set.
	\begin{coro}\label{coro:ineq_for_idempotents}
		For any pseudogroup $A$ and $a, b \in A$, we have the following inequalities
		\begin{align*}
			(a \land 1 ) \land ( b \land 1) & \leqslant \src(a \land b), & (a \land 1) \land \src(a \land b) & \leqslant (b \land 1), \\
			(a \land 1) \land (b \land 1) & \leqslant \tgt(a \land b), & (a \land 1) \land \tgt(a \land b) & \leqslant (b \land 1).
		\end{align*}
		in $(E(A),\leqslant)$.
	\end{coro}
	\begin{proof}
		By \cref{coro:part_bij_are_fin_complete}, it suffices to check these inequalities hold for all partial bijections $a, b \in \I(X)$ for a set $X$.  Since the terms involved are all idempotents, which in this case means a restriction of the identity on $X$, we can identify a term with its source as a subset of $X$, and so to show an inequality it suffices to show an inclusion between subsets.  The top-left inequality becomes the obvious inclusion:
		\begin{align*}
			 (a \cap \id_X) \cap (b \cap \id_X) = & \   \sett{x \in X}{x \in \src(a) , \, a(x) = x} \cap \sett{x \in X}{x \in \src(b), \, b(x) = x}, \\
			= & \ \sett{x \in X }{x \in \src(a), \, x \in \src(b), \, a(x) = x = b(x)}, \\
			\subseteq & \  \sett{x \in X}{x \in \src(a), \, x \in \src(b), \, a(x) = b(x)} = \src(a \cap b),
		\end{align*}
		while the top-right inequality is equally obvious:
		\begin{align*}
			& \ (a \cap \id_X) \cap \src(a \cap b) \\
			=& \ \sett{x \in X}{x \in \src(a), \, a(x) = x} \cap \sett{x \in X}{x \in \src(a), \, x \in \src(b), \, a(x) = b(x)}, \\
			= & \ \sett{x \in X}{x \in \src(a), \, x \in \src(b) , \, a(x) = b(x) = x}, \\
			\subseteq & \ \sett{x \in X}{x \in \src(b), \, b(x) = x} = ( b \cap \id_X).
		\end{align*}
		The bottom two inequalities follow from the top two inequalities by substituting $a^{-1}$ for $a$ and $b^{-1}$ for $b$, using the fact that, as an idempotent, $(a \land 1) = (a \land 1)^{-1} = (a^{-1} \land 1)$, and similarly $(b \land 1) = (b^{-1} \land 1)$. 
	\end{proof}
	\section{Faithful representation of pseudogroups}\label{sec:localic_rep}
	When the frame of idempotents $E(A)$ of a pseudogroup $A$ is not spatial, we cannot perform the same argument from \cref{subsec:subquotient}.  But, if we embrace point-free topology, as embodied by locale theory, there is no real loss: we can still faithfully represent any pseudogroup using partial homeomorphisms of a locale.  The idea is that we perform the same construction as in \cref{df:underlying_set}, except without specific reference to the points of $E(A)$, although the intuition will remain the same.

	We will explicitly construct a \emph{presentation} of the locale $Y$ in question, and describe partial homeomorphisms of $Y$ in terms of this presentation.  To facilitate our exposition, we therefore open by recalling both aspects of frame and locale theory that we will use in \cref{subsec:locale_of_subqts} -- the logical perspective on frames, as Lindenbaum-Tarski algebras for an infinitary propositional logic, and the topological perspective, where frames are treated earnestly as a point-free alternative to topological spaces.
	\subsection{Frame presentations}\label{subsec:frame_pres}
	As (infinitary) algebraic structures, frames can be presented in terms of \emph{generators} and \emph{relations}.  By the generators $G$, we mean any set of variable names -- these play the role of a prebasis from topology.  A frame presented by the set of generators $G$, with no relations, is the uniquely determined frame $[G]$ with the universal property that frame homomorphisms $v^\ast \colon [G] \to A$ correspond to functions $\underline{v} \colon G \to A$; for example, the free frame on one generator is concretely the frame of opens on the Sierpi\'nski space. (See \cite{johnstone_inv} for a comparison on various constructions of $[G]$.)
	
	A set of relations $R$ consists of equations $\phi = \psi$ between terms $\phi, \psi$ which have been constructed using generators from $G$ and the frame operations $\{0,1,\land,\bigvee\}$, and so both $\phi$ and $\psi$ can be interpreted as elements in the free frame $[G]$.  Frame homomorphisms from the presented frame $[G \mid R]$ to another frame $A$ correspond to functions $\underline{v} \colon G \to A$ such that the extended map $v^\ast \colon [G] \to A$ sends $\phi$ and $\psi$ to the same element in $A$.  Thus, $[G \mid R]$ is concretely the quotient frame of $[G]$ by the relations $R$.  Note that an inequality between terms can equivalently be expressed by the equality $\phi \land \psi = \phi$, and so for convenience we will often treat inequalites between terms as relations as well.
	
	A frame presentation can be understood as a theory of \emph{propositional geometric logic}, a logic which allows for finite conjunctions and infinitary disjunctions.  The set of generators $G$ is the set of propositional variables, i.e.\ the language of the theory, and the relations are the axioms (see \cite[\S D1]{elephant} for the complete syntax of propositional geometric logic).  The presented frame $[G \mid R]$ is then the Lindenbaum-Tarski algebra for the propositional theory encoded by $(G,R)$.  For this reason, we will write the inequalities in our set of relations using the logical entailment symbol `$\vdash$'.  From this perspective, a frame homomorphism $[G_1\mid R_1] \to [G_2 \mid R_2]$ represents an \emph{interpretation} of theories, where the generators in $G_1$ are assigned an interpretation as geometric formulae over the new language $G_2$, as represented by a function $\underline{v} \colon G_1 \to [G_2]$.  Thus, to show that such an interpretation $\underline{v} \colon G_1 \to [G_2]$ yields a well-defined frame homomorphism $[G_1 \mid R_1] \to [G_2 \mid R_2]$, we must show that, for each relation $\phi \vdash \psi$ contained in $R_1$, the image $v^\ast(\phi) \vdash v^\ast(\psi)$ is derivable from the theory $R_2$ (using the rules of geometric logic, see \cite[Definition D1.3.1]{elephant}).  This is how we will present the constructions in \cref{subsec:locale_of_subqts} -- as formal derivations in the sequent calculus of propositional geometric logic (though with many of the trivial deductions truncated).
	\begin{ex}
		Let $X$ be a set.  The frame $\P(X)$ can be presented by adding a generator $[x]$ for each element $x \in X$ and the relations $[x] \land [y] \vdash \bot$, for all pairs $x \neq y$, and $\top \vdash \bigvee_X [x]$.  Here, the generator $[x]$ represents the singleton open $\{x\} \in \P(X)$.
	\end{ex}
	\begin{ex}\label{ex:locale_self_prstn}
		Let $A$ be any frame.  Then $A$ can be trivially presented in terms of itself as follows.  For each element $a \in A$, we introduce a generator $\nbhd{a}$, subject to the following relations:
		\begin{enumerate}
			\item whenever $a \leqslant b$ in $A$, we add the relation $\nbhd{a} \vdash \nbhd{b}$,
			\item we add the relations $\nbhd{a} \land \nbhd{b} \vdash \nbhd{(a \land b)}$, for all $a, b \in A$, and $\top \vdash \nbhd{1}$, 
			\item and finally the relations $\nbhd{\left(\bigvee_I a_i\right)} \vdash \bigvee_I \nbhd{a}_i$, for all subsets $\sett{a_i}{i \in I} \subseteq A$, and $\nbhd{0} \vdash \bot$,
		\end{enumerate}
		These relations ensure that, for instance, the meet of two generators is equivalent to their meet in $A$, and similarly for the other frame operations.  Thus, the presented frame is isomorphic to $A$.
	\end{ex}
	\subsection{Locales}\label{subsec:locale_prelims}
	Recall that each a continuous map $f \colon X \to Y$ between spaces yields a frame homomorphism $f^\ast \colon \opens(Y) \to \opens(X)$.  For this reason, it is natural to consider the formal opposite of the category of frames and frame homomorphisms which, to emphasise the change in perspective, we call the category of \emph{locales} and \emph{locale morphisms}.  Within the category of locales, many of our intuitions from topology hold, with the additional benefit that the axiom of choice can generally be avoided (see, for example, the localic version of Tychonoff's theorem \cite{johnstone_tychonoff}).  See the later chapters of \cite{picado_pultr} or \cite{johnstone_stone} for this philosophy in action.  Because of this close comparison with topology, we call an isomorphism in the category of locales (equivalently, frames) a \emph{homeomorphism}.
	
	Moreover, many concepts from topology are readily generalised to locales.  For instance, the regular monomorphisms in the category of locales are called \emph{sublocales}, and are analogous to subspaces; concretely, a sublocale inclusion $B \subseteq A$ corresponds to a surjective frame homomorphism $A \twoheadrightarrow B$.  Each element of a locale $a \in A$ induces an \emph{open sublocale} $\langle a \rangle \subseteq A$, generalising the notion of an open subspace, and there is an isomorphism between $A$ and the lattice of open sublocales $B \subseteq A$, where the latter are ordered by sublocale inclusion (\cite[\S II.6]{picado_pultr}).  If we are given a presentation $[G \mid R]$ for the locale $A$, so that the element $a \in A$ corresponds to a term $\phi \in [G]$, then a presentation for the open sublocale $\langle a \rangle \subseteq A$ is easily obtained by simply adding the relation $\top \vdash \phi$ to the presentation of $A$.
	\subsection{The locale of subquotients}\label{subsec:locale_of_subqts}
	We now explicitly describe a presentation for a locale $Y$ and a pseudogroup embedding $\rho \colon A \hookrightarrow \H(Y)$, thus completing the proof of \cref{mainthm:no_pts}.  By construction, those points of $Y$ that do exist correspond to the pairs $(x,\eqr{a})$ found in \cref{df:underlying_set}.
	\begin{df}\label{df:presentation_of_Y}
		Let $Y$ denote the locale with the following presentation.
		\begin{enumerate}
			\item\label{enum:relations_from_EA} First, we include a copy of the trivial presentation of $E(A)$, i.e.\ we add a generator $\nbhd{e}$ for each idempotent $e \in E(A)$, and the relations from \cref{ex:locale_self_prstn}, e.g.\ $\nbhd{e} \land \nbhd{f} \vdash \nbhd{(e \land f)}$, etc.
			\item\label{enum:relations_for_eqr} Next, we add for each element $a \in A$ a generator $\eqr{a}$, and the relations $\eqr{a} \land \eqr{a}' \vdash \nbhd{\tgt(a\land a')}$ and $\eqr{a} \land \nbhd{\tgt(a \land a')} \vdash \eqr{a}'$, for all pairs $a, a' \in  A$.
			\item\label{enum:relations_big_join} Finally, we add the relation $\nbhd{e} \vdash \bigvee_{\tgt(a) \leqslant e} \eqr{a}$, for each $e \in E(A)$.
		\end{enumerate}
	\end{df}
	Note that, since $Y$ contains a copy of the trivial presentation of $E(A)$, the identity assignment $\nbhd{e} \mapsto \nbhd{e}$ automatically yields a frame homomorphism which we denote by $q^\ast \colon E(A) \to Y$ (so that the corresponding locale morphism is denoted by $q \colon Y \to E(A)$).  We will later show in \cref{prop:q_has_section} that $q^\ast$ is injective.
	\begin{df}
		For an element $b \in A$, we define a homeomorphism
		\[
		\rho_b \colon \lrangle{\nbhd{\src(b)}} \to \lrangle{\nbhd{\tgt(b)}}
		\]
		between the open sublocales $\lrangle{\nbhd{\src(b)}}, \lrangle{\nbhd{\tgt(b)}} \subseteq Y$ by setting the action of the inverse image $\rho_b^\ast$ on generators of $\lrangle{\nbhd{\tgt(b)}}$ as:
		\begin{align*}
			\nbhd{e} & \mapsto \nbhd{(b^{-1}eb)}, \\
			\eqr{a} & \mapsto \eqr{b^{-1}a},
		\end{align*}
		for all $e \in E(A)$ and $a \in A$.
	\end{df}
	\begin{prop}
		For each $b \in A$, the map $\rho_b$ is a well-defined homeomorphism of locales.
	\end{prop}
	\begin{proof}
		First, we show that $\rho_b^\ast \colon \lrangle{\nbhd{\tgt(b)}} \to \lrangle{\nbhd{\src(b)}}$ is a well-defined frame homomorphism, i.e.\ the action of $\rho_b^\ast$ on generators respects relations.  We check each relation from \cref{df:presentation_of_Y} in turn.
		\begin{enumerate}
			\item It is straightforward to show that $\rho_b^\ast$ preserves those relations coming from the trivial presentation on $E(A)$, i.e.\ the relations in \cref{df:presentation_of_Y}\cref{enum:relations_from_EA}. For instance, for the relation $\nbhd{e} \vdash \nbhd{f}$ of $\lrangle{\nbhd{\tgt(b)}}$, where $e \leqslant f$, we have that $b^{-1}eb \leqslant b^{-1}fb$, and so the image of the relation, $\rho_b^\ast(\nbhd{e}) = \nbhd{(b^{-1}eb)} \vdash \nbhd{(b^{-1}fb)} = \rho_b^\ast(\nbhd{f})$, is satisfied in $\lrangle{\nbhd{\src(b)}}$.
			\item The relation $\eqr{a} \land \eqr{a}' \vdash \nbhd{\tgt(a \land a')}$ is automatically preserved, since we have that
			\begin{align*}
				\rho_b^\ast(\eqr{a}) \land \rho_b^\ast(\eqr{a'}) = \eqr{b^{-1}a} \land \eqr{b^{-1}a'} & \vdash \nbhd{\tgt(b^{-1}a \land b^{-1}a')}, \\
				& = \nbhd{((b^{-1}a \land b^{-1}a')(b^{-1}a \land b^{-1}a')^{-1})}, \\
				& = \nbhd{(b^{-1}(a \land a')(a \land a')^{-1}b)} = \rho_b^\ast(\nbhd{\tgt(a \land a')})
			\end{align*}
			in $\lrangle{\nbhd{\src(b)}}$.  Similarly, the relation $\eqr{a} \land \nbhd{\tgt(a \land a')} \vdash \eqr{a}'$ is also preserved by $\rho_b^\ast$ since, by the above calculation, we have that
			\[
			\rho_b^\ast(\eqr{a}) \land \rho_b^\ast(\nbhd{\tgt(a \land a')}) = \eqr{b^{-1}a} \land \nbhd{\tgt(b^{-1}a \land b^{-1}a')} \vdash \eqr{b^{-1}a'} = \rho_b^\ast(\eqr{a}').
			\]
			\item Next we consider the relation $\nbhd{e} \vdash \bigvee_{\tgt(a) \leqslant e} \eqr{a}$.  First observe that, since $\lrangle{\nbhd{\src(b)}}$ satisfies the relation $\top \vdash \nbhd{\src(b)}$, we have that
			\begin{align*}
				\eqr{a}' & \vdash \eqr{a}' \land \nbhd{\src(b)} \land \nbhd{\tgt(a')} \land \nbhd{\src(b)}, \\
				& \vdash \eqr{a}' \land \nbhd{(\src(b) \land \tgt(a') \land \src(b))}, \\
				& = \eqr{a}' \land \nbhd{(b^{1}ba'{a'}^{-1}b^{-1}b)}, \\
				& = \eqr{a}' \land \nbhd{\tgt(b^{-1}ba')}, \\
				& = \eqr{a}' \land \nbhd{\tgt(b^{-1}ba' \land a')}, \\
				& \vdash \eqr{b^{-1}ba'},
			\end{align*}
			is valid in $\lrangle{\nbhd{\src(b)}}$, for all $a' \in A$, where we have used that $b^{-1}ba' \leqslant a'$.  Note also that if $\tgt(a') = a'{a'}^{-1} \leqslant b^{-1}eb$, then $\tgt(ba') = ba'{a'}^{-1}b^{-1} \leqslant bb^{-1}ebb^{-1} \leqslant e$.  Thus, we have the relation
			\begin{align*}
			\rho_b^\ast(\nbhd{e}) = \nbhd{(b^{-1}eb)} & \textstyle \vdash \bigvee_{\tgt(a') \leqslant b^{-1}eb} \eqr{a}', \\
			 & \textstyle \vdash \bigvee_{\tgt(a') \leqslant b^{-1}eb} \eqr{b^{-1}ba'} , \\
			  & \textstyle = \bigvee_{\tgt(a') \leqslant b^{-1}eb} \rho_b^\ast(\eqr{ba'}) , \\
			  & \textstyle \vdash \bigvee_{\tgt(a) \leqslant e} \rho_b^\ast(\eqr{a}) 
			\end{align*}
			is satisfied in $\lrangle{\nbhd{\src(b)}}$, as desired.
			\item Finally, the relation $\top \vdash \nbhd{\tgt(b)}$ of $\lrangle{\nbhd{\tgt(b)}}$ is preserved since the relation $\top \vdash \nbhd{\src(b)} = \nbhd{(b^{-1}b)} = \nbhd{(b^{-1}\tgt(b)b)} = \rho_b^\ast(\nbhd{\tgt(b)})$ is satisfied in $\lrangle{\nbhd{\src(b)}}$.
		\end{enumerate}
		
		Next, we claim that $\rho_{b^{-1}}$ is the inverse to $\rho_b$.  By symmetry, it suffices to show that $\rho_{b^{-1}}^\ast \rho_b^\ast$ sends each generator of $\lrangle{\nbhd{\tgt(b)}}$ to an equivalent element.  For the generator $\nbhd{e}$, we have one relation $\rho_{b^{-1}}^\ast \rho_b^\ast(\nbhd{e}) = \nbhd{(bb^{-1}ebb^{-1})} \vdash \nbhd{e}$ immediately from the fact that $bb^{-1}ebb^{-1} \leqslant e$; the converse relation holds since, modulo the relations of $\lrangle{\nbhd{\tgt(b)}}$, we have that 
		\[\nbhd{e} \vdash \nbhd{e} \land \nbhd{\tgt(b)} \vdash \nbhd{(e \land \tgt(b))} = \nbhd{(ebb^{-1})} = \nbhd{(bb^{-1}ebb^{-1})}.\]
		For the generator $\eqr{a}$, using that $b^{-1}ba \land a = b^{-1}ba$, we have the derivation
		\begin{align*}
			\rho_{b^{-1}}^\ast \rho_b^\ast(\eqr{a}) = \eqr{b^{-1}ba} & \vdash \eqr{b^{-1}ba} \land \nbhd{\tgt(b^{-1}ba)}, \\
			& = \eqr{b^{-1}ba} \land \nbhd{\tgt(b^{-1}ba \land a)}, \\
			& \vdash \eqr{a},
		\end{align*}
		while the converse derivation is given by
		\begin{align*}
			\eqr{a} & \vdash \eqr{a} \land \nbhd{\tgt(a)} \land \nbhd{\tgt(b)}, \\
			& \vdash \eqr{a} \land \nbhd{(\tgt(a) \land \tgt(b))}, \\
			& = \eqr{a} \land \nbhd{(bb^{-1}aa^{-1}bb^{-1})}, \\
			& = \eqr{a} \land \nbhd{\tgt(bb^{-1}a)}, \\
			& = \eqr{a} \land \nbhd{\tgt(a \land bb^{-1}a)}, \\
			& \vdash \eqr{bb^{-1}a} = \rho_{{b}^{-1}}^\ast \rho_b^\ast(\eqr{a}),
		\end{align*}
		thus completing the proof.
	\end{proof}
	\begin{thm}\label{thm:localic_rho}
		For any pseudogroup $A$, the assignment $\rho \colon A \to \Homeo(Y)$ is an injective pseudogroup homomorphism.
	\end{thm}
	Once again, we break down the proof into digestible chunks.
	\begin{prop}\label{prop:locale_semigroup_homo}
		The assignment $\rho \colon A \to \Homeo(Y)$ is a unit-preserving semigroup homomorphism.
	\end{prop}
	\begin{proof}
		Clearly, $\rho_1^\ast$ acts as the identity on generators, and its domain is the entirety of the locale $Y$ by the relation $\top \vdash \nbhd{1} = \nbhd{\src(1)}$.  Thus, the unit is preserved.
		
		Recall that the domain of $\rho_b$, that is the sublocale $\lrangle{\nbhd{\src(b)}}\subseteq Y$, is presented by adding the relation $\top \vdash \nbhd{\src(b)}$ to the presentation given in \cref{df:presentation_of_Y}, and similarly for $\rho_{b'}$.  The domain of $\rho_{b'} \rho_{b}$ is the (open) sublocale of $\lrangle{\nbhd{\src(b)}}$ induced by the inverse image of $\lrangle{\nbhd{\src(b')}}$, the domain of $\rho_{b'}$, under the locale morphism $\rho_{b^{-1}}$, which can be concretely presented by adding the relation $\top \vdash \rho_{{b}^{-1}}^\ast(\nbhd{\src(b')})$ to the presentation of $\lrangle{\nbhd{\src(b')}}$.  Using that $\rho_{{b}^{-1}}^\ast(\nbhd{\src(b')}) = \nbhd{(b^{-1}{b'}^{-1}b'b)} = \nbhd{\src(b'b)}$, we have that the domain of $\rho_{b'}\rho_b$ is contained in the sublocale $\lrangle{\nbhd{\src(b'b)}}$.  The converse inclusion is ensured since $\nbhd{\src(b'b)} \vdash \nbhd{\src(b)}$.  Thus, the domains of $\rho_{b'}\rho_b$ and $\rho_{b'b}$ coincide, and it is clear that they act identically on generators, and so $\rho_{b'}\rho_b = \rho_{b'b}$ as desired.
	\end{proof}
	Just as in \cref{coro:monotone}, \cref{prop:locale_semigroup_homo} entails that $\rho$ is monotone with respect to the partial orders.
	\begin{lem}\label{lem:localic_rho_preserves_meets}
		The semigroup homomorphism $\rho \colon A \to \Homeo(Y)$ preserves binary meets.
	\end{lem}
	\begin{proof}
		Let $b, b' \in A$.  We aim to show that $\rho_b \cap \rho_{b'} = \rho_{b \land b'}$.  By the monotonicity of $\rho$, we have one inclusion $\rho_{b \land b'} \subseteq \rho_b \cap \rho_{b'}$, i.e.\ the domain of $\rho_{b \land b'}$ is included in the domain of $\rho_b \cap \rho_{b'}$ and the two maps agree when restricted to this sublocale.  It suffices to show the converse inclusion of domains.  In terms of frame presentations, we must show that a presentation of the domain of $\rho_b \cap \rho_{b'}$ proves the defining relation $\top \vdash \nbhd{\src(b \land b')}$ of the sublocale $\lrangle{\nbhd{\src(b \land b')}} \subseteq Y$.
		
		Recall that the domain of $\rho_b \cap \rho_{b'}$ is the largest sublocale of $Y$ on which $\rho_b$ and $\rho_{b'}$ agree.  Expressed concretely, the domain of $\rho_b \cap \rho_{b'}$ is the coequalizer of the frame homomorphisms
		\[
		\begin{tikzcd}[column sep=large]
			\lrangle{\nbhd{\tgt(b)}} \cap \lrangle{\nbhd{\tgt(b')}} \ar[shift left=2]{r}{\rho_b^\ast |_{\lrangle{\nbhd{\tgt(b')}}}} \ar[shift right=2]{r}[']{\rho_{b'}^\ast |_{\lrangle{\nbhd{\tgt(b)}}}} & \lrangle{\nbhd{\src(b)}} \cap \lrangle{\nbhd{\src(b')}},
		\end{tikzcd}
		\]
		which can be expressed by adding to the presentation of $Y$ the relations
		\begin{align*}
			\top & \vdash \nbhd{\src(b)} \land \nbhd{\src(b')}, \\
			\rho_{b}^\ast(\nbhd{e}) = \nbhd{(b^{-1}eb)} \dashv & \vdash \nbhd{({b'}^{-1}eb')} = \rho_{b'}^\ast(\nbhd{e}), \\
			\rho_{b}^\ast(\eqr{a}) = \eqr{b^{-1}a} \dashv & \vdash \eqr{{b'}^{-1}a} = \rho_{b'}^\ast(\eqr{a})
		\end{align*}
		for all $e \in E(A)$ and $a \in A$.  (Here, the double sequent $\dashv \, \vdash$ means that the sequent holds in both directions.)
		
		Then, using the above relations in addition to those in \cref{df:presentation_of_Y}, we have the following derivation:
		\begin{align*}
			\top & \vdash \nbhd{\src(b)} \land \nbhd{\src(b')}, \\
			& \vdash \textstyle  \bigvee \lrset{\eqr{a} \land \eqr{a}'}{\tgt(a) \leqslant \src(b), \, \tgt(a') \leqslant \src(b')}, \\
			& \vdash \textstyle \bigvee \lrset{\nbhd{\tgt(a \land a')}}{\tgt(a) \leqslant \src(b), \, \tgt(a') \leqslant \src(b')}, \\
			& \vdash \nbhd{\tgt(b^{-1}\land {b'}^{-1})} = \nbhd{\src(b \land b')},
		\end{align*}
		completing the proof.
	\end{proof}
	\begin{lem}\label{lem:localic_rho_preserves_joins}
		The semigroup homomorphism $\rho \colon A \to \Homeo(Y)$ preserves joins of compatible subsets.
	\end{lem}
	\begin{proof}
		Let $\sett{b_i}{i \in I} \subseteq A$ be a jointly compatible subset.  By monotonicity, one inclusion, $\bigcup_I \rho_{b_i} \subseteq \rho_{\bigvee_I b_i}$, is automatic, and so it remains to show that the domain of $\rho_{\bigvee_I b_i}$, i.e.\ the sublocale $\lrangle{\nbhd{\src\left(\bigvee_I b_i\right)}} \subseteq Y$, is contained in the domain of $\bigcup_I \rho_{b_i}$, i.e.\ the join $\bigcup_I \lrangle{\nbhd{\src(b_i)}}$, where the join is taken in the complete lattice of sublocales of $Y$ (see \cite[\S III, Theorem 3.2.1]{picado_pultr}).  This is immediate since
			\[ \textstyle
			\lrangle{\nbhd{\src\left(\bigvee_I b_i\right)}} = \lrangle{\nbhd{\bigvee_I \src(b_i)}} = \bigcup_I \lrangle{\nbhd{\src(b_i)}},
			\]
		where the first equality follows from \cref{prop:basic_props}\cref{enum:src_of_join}, while for the second equality we use \cite[\S III, Proposition 6.1.5]{picado_pultr}.
	\end{proof}
	\begin{prop}\label{prop:q_has_section}
		The locale morphism $q \colon Y \to E(A)$ has a section $s \colon E(A) \to Y$.
	\end{prop}
	\begin{proof}
		We define a frame homomorphism $s^\ast \colon Y \to E(A)$ on generators as follows:
		\[
		\nbhd{e} \mapsto \nbhd{e}, \qquad \eqr{a} \mapsto \nbhd{(a \land 1)}.
		\]
		It is clear that, if $s^\ast$ yields a well-defined frame homomorphism, then $s^\ast q^\ast(\nbhd{e}) = \nbhd{e}$ for all $e \in E(A)$, and so $s$ defines a section of $q$.
		
		We must therefore show that $s^\ast$ yields a well-defined frame homomorphism, i.e.\ that $s^\ast$ respects the relations of $Y$.  For those relations derived from the trivial presentation of $E(A)$, i.e.\ the relations in \cref{df:presentation_of_Y}\cref{enum:relations_from_EA}, this is immediate.  The relations described in \cref{df:presentation_of_Y}\cref{enum:relations_for_eqr}, under the action $s^\ast$, become the inequalities $(a \land 1) \land (a' \land 1) \leqslant \tgt(a \land a')$ and $(a \land 1) \land \tgt(a \land a') \leqslant (a' \land 1)$ in $E(A)$, which we know to hold via \cref{coro:ineq_for_idempotents}.
		
		Finally, since $e = (e \land 1) = \tgt(e)$ for every idempotent $e \in E(A)$, we have that $e \leqslant \bigvee_{\tgt(a) \leqslant e} (a \land 1)$, demonstrating that $s^\ast$ preserves the relation in \cref{df:presentation_of_Y}\cref{enum:relations_big_join} and thus completing the proof.
	\end{proof}
	\begin{rem}
		Intuitively, the locale morphism $s \colon E(A) \to Y$ describes, in locale-theoretic terms, the map that sends a point $x$ of $E(A)$ to the pair $(x,\eqr{1})$, i.e.\ a point of the set constructed in \cref{df:underlying_set}.
	\end{rem}
	Hence, there is a frame homomorphism $s^\ast \colon Y \to E(A)$ for which the composite $s^\ast q^\ast \colon E(A) \to E(A)$ is the identity.  It follows that $q^\ast \colon E(A) \to Y$ is injective.
	\begin{coro}\label{coro:localic_rho_preserves_strict_order}
		The semigroup homomorphism $\rho \colon A \to \Homeo(Y)$ preserves the strict natural order.
	\end{coro}
	\begin{proof}
		Since $q^\ast \colon E(A) \to Y$ is injective, if $e < f$ in $E(A)$, then $\nbhd{e} < \nbhd{f}$ in $Y$, and so $\lrangle{\nbhd{e}} \subsetneq \lrangle{\nbhd{f}}$ as sublocales of $Y$.  If $b < b'$ in $A$, then $\src(b) < \src(b')$.  Thus, the domain of $\rho_b$ is strictly contained in the domain of $\rho_{b'}$, and so $\rho$ preserves the strict natural order.
	\end{proof}
	\begin{coro}\label{coro:localic_rho_is_inject}
		The pseudogroup homomorphism $\rho \colon A \to \Homeo(Y)$ is injective.
	\end{coro}
	\begin{proof}
		By combining \cref{coro:localic_rho_preserves_strict_order}, \cref{lem:localic_rho_preserves_meets} and \cref{lem:strict_ineq_implies_inject}.
	\end{proof}
	Together, \cref{prop:locale_semigroup_homo}, \cref{lem:localic_rho_preserves_meets}, \cref{lem:localic_rho_preserves_joins} and \cref{coro:localic_rho_is_inject} complete the proof of \cref{thm:localic_rho}.  Thus, even when the pseudogroup $A$ lacks enough points, there is still a sense, in the context of point-free topology, that $A$ is a `pseudogroup of transformations'.
	% References
	\printbibliography
\end{document}